\documentclass[reqno]{amsart}
\usepackage{amssymb,latexsym,amsmath,amsthm,enumerate,amsbsy}
\usepackage[mathscr]{eucal}
\usepackage{framed,color,graphicx}
\usepackage{mathrsfs}
\usepackage[all]{xy}
\usepackage{tikz}
\usepackage{cite}
\usepackage{longtable}
\usepackage{ltxtable}
\usepackage{subcaption}
\usepackage{graphicx}
\usepackage{subfloat}
\usepackage{subfig}
\usepackage{tikz}
\usepackage{caption} 
\usepackage{etoolbox} 
\usepackage{tabularx}
\usepackage{eqparbox}
\usepackage{array}

\usepackage{enumitem}

\makeatletter
\@namedef{subjclassname@2020}{\textup{2020} Mathematics Subject Classification}
\makeatother

\usetikzlibrary{positioning,decorations.pathreplacing,patterns,decorations.pathmorphing}
\tikzset{%
element/.style={draw, shape=circle, fill=white, inner sep=1.4pt}
}

\DeclareSymbolFont{bbold}{U}{bbold}{m}{n}
\DeclareSymbolFontAlphabet{\mathbbold}{bbold}

\theoremstyle{plain}
\newtheorem{thm}{Theorem}[section]
\newtheorem{lem}[thm]{Lemma}

\newtheorem{cor}[thm]{Corollary}
\newtheorem{pro}[thm]{Proposition}

\newtheorem{problem}[thm]{Problem}

\theoremstyle{definition}

\newtheorem{remark}[thm]{Remark}

\newcounter{eqbasis}
\newcommand{\bt}{\mathbf{t}}
\newcommand{\bu}{\mathbf{u}}
\newcommand{\bv}{\mathbf{v}}
\newcommand{\bw}{\mathbf{w}}

\begin{document}
\title[The finite basis problem for matrix semirings]
{The finite basis problem for matrix semirings $\mathbf{M}_n(S_7)$}

\author{Jun Jiao}
\address{School of Mathematics, Northwest University, Xi'an, 710127, Shaanxi, P.R. China}
\email{jjunjiao@163.com}

\author{Miaomiao Ren}
\address{School of Mathematics, Northwest University, Xi'an, 710127, Shaanxi, P.R. China}
\email{miaomiaoren@yeah.net}

\subjclass[2020]{16Y60, 03C05, 08B15}
\keywords{semiring, matrix, variety, identity, finite basis problem}

\begin{abstract}
We first prove that two matrix semirings $\mathbf{M}_n(S_1)$ and $\mathbf{M}_n(S_2)$ are equationally equivalent
whenever additively idempotent semirings $S_1$ and $S_2$ are equationally equivalent.
We then prove an embedding theorem for matrix semirings $\mathbf{M}_n(S)$ over an additively idempotent semiring $S$:
for all $n \geq 2$, $\mathbf{M}_n(S)$ embeds into $\mathbf{M}_{n+1}(S)$. This yields an ascending chain of varieties $\mathsf{V}(\mathbf{M}_2(S)) \leq \mathsf{V}(\mathbf{M}_3(S)) \leq \cdots$,
which is strictly ascending when $S$ is the two-element distributive lattice.

Finally, we show that every variety in the interval $[\mathsf{V}(S_c(abc)), \mathsf{V}(\mathbf{M}_n(S_7))]$ is nonfinitely based
(i.e., has no finite basis for its identities),
where $S_c(abc)$ is an eight-element flat semiring and
$S_7$ is the unique nonfinitely based three-element additively idempotent semiring.
Consequently, $\mathbf{M}_n(S_7)$ is nonfinitely based,
yielding an ascending chain $\mathsf{V}(\mathbf{M}_2(S_7)) \leq \mathsf{V}(\mathbf{M}_3(S_7)) \leq \cdots$;
moreover, every variety in $[\mathsf{V}(S_7), \mathsf{V}(\mathbf{M}_n(S_7))]$ is also nonfinitely based,
and this interval contains at least countably infinitely many distinct varieties.
\end{abstract}

\maketitle

\section{Introduction and preliminaries}
By an \emph{additively idempotent semiring} (ai-semiring for short) we mean an algebra $(S, +, \cdot)$ such that
the additive reduct $(S, +)$ is a commutative idempotent semigroup,
the multiplicative reduct $(S, \cdot)$ is a semigroup,
and the distributive laws
\[
x(y+z) \approx xy + xz, \quad (x+y)z \approx xz + yz
\]
hold.

Let $S$ be an ai-semiring. Define a binary relation $\leq$ on $S$ by
\[
a \leq b \Leftrightarrow a + b = b.
\]
Then $\leq$ is a partial order compatible with both addition and multiplication;
hence an ai-semiring is sometimes called a \emph{semilattice-ordered semigroup} \cite{kp}.
Whenever an order on an ai-semiring is mentioned, it always
refers to the order defined above.

The class of ai-semirings includes well-known examples such as
max-plus algebras~\cite{aei}, the flat extensions of groups~\cite{jackson08}, and distributive lattices~\cite{burris1981}.
These and other similar algebras have proven to be a powerful algebraic framework with applications spanning multiple disciplines.
They form the algebraic foundation of tropical geometry~\cite{ms} and tropical linear algebra~\cite{but}, and have also found important applications in theoretical computer science~\cite{go} and information science~\cite{gl}.

One natural way to construct new ai-semirings from a given one is to consider matrix semirings over it.
More specifically,
let $S$ be an ai-semiring, let $n \geq 2$ be an integer, and let $\mathbf{M}_n(S)$ denote the set of all $n \times n$ matrices over $S$.
Then $\mathbf{M}_n(S)$ forms an ai-semiring under the usual matrix addition and multiplication:
for any $A = [a_{ij}]_{n \times n}$ and $B = [b_{ij}]_{n \times n} $ in $\mathbf{M}_n(S)$,
\[
A + B = [a_{ij} + b_{ij}]_{n \times n}, \quad
A \cdot B = \left[ \sum_{k=1}^n a_{ik} b_{kj} \right]_{n \times n}.
\]
Jiao and Ren~\cite{jr2025} observed that $S$ can be embedded into $\mathbf{M}_n(S)$,
and consequently $\mathsf{V}(S)$ is a subvariety of $\mathsf{V}(\mathbf{M}_n(S))$.
We shall show later in this paper (Theorem~\ref{thm1}) that $\mathbf{M}_n(S)$ can be embedded into $\mathbf{M}_{n+1}(S)$;
therefore, $\mathsf{V}(\mathbf{M}_n(S))$ is a subvariety of $\mathsf{V}(\mathbf{M}_{n+1}(S))$.

A \emph{variety} of ai-semirings is a class of ai-semirings
closed under taking subalgebras, homomorphic images, and arbitrary direct products.
By Birkhoff's celebrated theorem~\cite{birkhoff1935},
a class of ai-semirings is a variety if and only if it is an \emph{equational class};
that is, the class of all ai-semirings satisfying a certain set of identities.
A variety is \emph{finitely based} if it possesses a finite equational basis
(i.e., it can be defined by finitely many identities); otherwise, it is \emph{nonfinitely based}.

For an ai-semiring $S$, let $\mathsf{V}(S)$ denote the variety generated by $S$; that is,
the smallest variety containing $S$.
Equivalently, $\mathsf{V}(S)$ consists of all homomorphic images of subalgebras of direct products of copies of $S$
(or simply, the class of all algebras that satisfy every identity that holds in $S$).
Then $S$ and $\mathsf{V}(S)$ satisfy precisely the same identities.
We say that $S$ is finitely based (resp., nonfinitely based)
if the variety $\mathsf{V}(S)$ generated by $S$ is finitely based (resp., nonfinitely based).
The finite basis problem for a class of ai-semirings, one of the central problems in universal algebra,
concerns the classification of its members
according to whether they are finitely based.

Over the past two decades, the finite basis problem for ai-semirings has attracted considerable attention and been extensively studied; see, for example,~\cite{dol07, sr, jrz, dgv25, gjr25, gv2501, zrc, rlzc, rlyc, ryy, rz, yrzs}.
Dolinka~\cite{dol07} found the first example of a nonfinitely based finite ai-semiring.
Dolinka, Gusev and Volkov~\cite{dgv25} solved
the finite basis problem for the endomorphism semirings of finite semilattices.
Shao and Ren~\cite{sr} showed that every algebra in the variety generated by all two-element ai-semirings is finitely based.
Jackson et al.~\cite{jrz} and Zhao et al.~\cite{zrc} provided a complete classification of all three-element ai-semirings with respect to the finite basis property.
It was shown that, up to isomorphism, $S_7$ is the unique nonfinitely based three-element ai-semiring; its addition and multiplication tables are presented in Table~\ref{tbs7}.

This remarkable fact distinguishes $S_7$ among all three-element ai-semirings, making it a uniquely intriguing object.
Despite its seemingly simple appearance, having only three elements and a commutative multiplication, the finite basis problem surrounding $S_7$ turns out to be surprisingly rich and nontrivial.
Consequently, a thorough understanding of its influence on other algebras is of central importance.
Indeed, from different perspectives,
\cite{gjr25, gjrz2, gv2301, gv2302, gv2501, wrz24, yue2026, yue2026b}
have demonstrated that $S_7$ can transmit the nonfinitely based property to many other finite ai-semirings,
further highlighting its special role.
These observations naturally lead to a broader question, which was explicitly posed by Gao et al.~\cite{gjrz2} (a manuscript in preparation, not publicly available):

\begin{problem}\label{prob26052301}
Is every finite ai-semiring whose variety contains $S_7$ nonfinitely based?
\end{problem}

\begin{table}[ht]
\centering
\caption{The Cayley tables of $S_7$} \label{tbs7}
\begin{tabular}{c|ccc}
$+$&1&$a$&$\infty$\\
\hline
$1$&1&$\infty$&$\infty$\\
$a$&$\infty$&$a$&$\infty$\\
$\infty$&$\infty$&$\infty$&$\infty$\\
\end{tabular}\qquad
\begin{tabular}{c|ccc}
$\cdot$&1&$a$&$\infty$\\
\hline
$1$&1&$a$&$\infty$\\
$a$&$a$&$\infty$&$\infty$\\
$\infty$&$\infty$&$\infty$&$\infty$\\
\end{tabular}
\end{table}

In another direction,
Dolinka~\cite{dolinka2009} first considered the finite basis problem for matrix semirings over an ai-semiring.
More specifically, he proved that the matrix semiring $\mathbf{M}_2(D_2)$,
where $D_2=\{0, 1\}$ denotes the two-element distributive lattice with operations
\[
0+0=0, 0+1=1+0=1, 1+1=1, 0\cdot 0=0, 0\cdot 1=1 \cdot 0=0, 1\cdot 1=1,
\]
is inherently nonfinitely based; that is, every locally finite variety
(i.e., a variety in which every finitely generated algebra is finite)
containing $\mathbf{M}_2(D_2)$ is nonfinitely based.
Consequently, $\mathbf{M}_n(D_2)$ is nonfinitely based for all $n\geq 2$.
Based on this, Jiao and Ren~\cite{jr2025} completed the classification of matrix semirings over two-element ai-semirings with respect to the finite basis property.
Recently,
Gusev and Volkov~\cite{gv2510, volkov2024} have been investigating
the same problem for upper triangular matrix semirings over $D_2$.

Inspired by these results,
the present paper focuses on the finite basis problem for the matrix semiring $\mathbf{M}_n(S_7)$ with $n \geq 2$.
We shall show that all of them are nonfinitely based.
In fact, we shall prove a stronger result: every variety contained in $\mathsf{V}(\mathbf{M}_n(S_7))$
that also contains $S_7$ is nonfinitely based,
thereby contributing to Problem~\ref{prob26052301}.

To this end, we recall some basic terminology and notation that will be used throughout the paper.
Let $X$ be a countably infinite set of variables, and let $X^+$ denote the free semigroup over $X$.
An \emph{ai-semiring term} over $X$ is a finite nonempty set of words in $X^+$.
We write a term as a formal sum of its elements; that is,
$\bw=\bu_1+\bu_2+\cdots+\bu_n$ means $\bw=\{\bu_1, \bu_2, \ldots, \bu_n\}$.
The order of the summands in the formal sum is irrelevant,
and repeated occurrences of the same word are identified with a single occurrence.
Two terms are equal if and only if their underlying sets coincide.

The collection of all terms over $X$, denoted by $P_f(X^+)$,
forms an ai-semiring under the usual term addition and multiplication.
By \cite[Theorem 2.5]{kp}, $P_f(X^+)$ is free in the variety of all ai-semirings over $X$.

An \emph{ai-semiring identity} (or simply an identity) over $X$ is a formal expression of the form
\[
\bu\approx \bv,
\]
where $\bu$ and $\bv$ are ai-semiring terms over $X$.
For an ai-semiring $S$ and an ai-semiring identity $\bu\approx \bv$,
we say that $S$ \emph{satisfies} $\bu\approx \bv$ (or $\bu\approx \bv$ \emph{holds} in $S$)
if $\varphi(\bu)=\varphi(\bv)$
for every semiring homomorphism $\varphi\colon P_f(X^+) \to S$.
Note that such homomorphism $\varphi$ is uniquely determined
by its values on $X$, since $P_f(X^+)$ is generated by $X$.

%

A \emph{flat semiring} is an ai-semiring whose multiplicative reduct contains a zero element $0$
and satisfies $a + b = 0$ for any two distinct elements $a, b \in S$.
For example, the three-element algebra $S_7$ is a flat semiring.
Such algebras have played an important role in the development of the finite basis problem for ai-semirings.
Jackson et al.~ \cite[Lemma 2.2]{jrz} observed that a semigroup with a zero element $0$
becomes a flat semiring with $0$ as the additive maximum element if and only if it is $0$-cancellative;
that is, for all $a, b, c \in S$,
\[
ab = ac \neq 0 \Rightarrow b = c \quad \text{and} \quad ba = ca \neq 0 \Rightarrow b = c.
\]

The following algebras constitute an important class of flat semirings.
Let $W$ be a nonempty subset of a free semigroup, and let $S(W)$ denote the set of all nonempty subwords of words in $W$ together with an additional symbol $0$.
Define a binary operation $\cdot$ on $S(W)$ by
\[
0 \cdot 0 = 0 \cdot \mathbf{u} = \mathbf{u} \cdot 0 = 0
\]
for all $\mathbf{u} \in S(W) \setminus \{0\}$, and for $\mathbf{u}, \mathbf{v} \in S(W) \setminus \{0\}$,
\[
\mathbf{u} \cdot \mathbf{v} =
\begin{cases}
\mathbf{u}\mathbf{v}, & \text{if } \mathbf{u}\mathbf{v} \in W^{\leq}, \\
0, & \text{otherwise}.
\end{cases}
\]
Then $(S(W), \cdot)$ is a $0$-cancellative semigroup, and consequently $S(W)$ becomes a flat semiring.
If $W$ consists of a single word $\mathbf{w}$, we simply denote $S(W)$ by $S(\mathbf{w})$.
When starting from a nonempty subset of a free commutative semigroup, the corresponding constructions are denoted by $S_c(W)$ and $S_c(\mathbf{w})$.

The paper is organized as follows.
In Section~2, we establish two foundational theorems on matrix semirings:
the first shows that matrix semirings preserve equational equivalence,
and the second provides an embedding of $\mathbf{M}_n(S)$ into $\mathbf{M}_{n+1}(S)$.
In Section~3, we apply a hypergraph-based criterion to prove that $\mathbf{M}_n(S_7)$ is nonfinitely based for all $n \geq 2$.
In Section~4, we show that the interval $[\mathsf{V}(S_7), \mathsf{V}(\mathbf{M}_2(S_7))]$ contains at least countably infinitely many distinct varieties.
In Section~5, we study the multiplicative reduct of $\mathbf{M}'_n(S_7)$ and establish its $5$-nilpotency.
Finally, Section~6 concludes the paper with a summary and some open problems.

\section{Two theorems on matrix semirings}
In this section, we present two theorems on matrix semirings.
The first concerns equational equivalence:
$\mathbf{M}_n(S_1)$ and $\mathbf{M}_n(S_2)$ are equationally equivalent whenever $S_1$ and $S_2$ are.
The second is an embedding theorem: $\mathbf{M}_n(S)$ embeds into $\mathbf{M}_{n+1}(S)$.
We begin with three auxiliary lemmas for the first theorem.

\begin{lem}\label{lem24}
Let $(S_i)_{i\in I}$ be a family of ai-semirings, and let $n\geq 2$ be an integer.
Then $\mathbf{M}_n\bigl(\prod_{i\in I}S_i\bigr)$ is isomorphic to $\prod_{i\in I}\mathbf{M}_n(S_i)$.
\end{lem}
\begin{proof}
For each matrix $A=[a_{st}]_{n\times n} \in \mathbf{M}_n\bigl(\prod_{i\in I}S_i\bigr)$,
we write $a_{st}=\left(a_{st}^{(i)}\right)_{i\in I}$ with $a_{st}^{(i)}\in S_i$.
Define a mapping
\[
\varphi\colon \mathbf{M}_n\left(\prod_{i\in I}S_i\right) \to \prod_{i\in I}\mathbf{M}_n(S_i), \quad A \mapsto \left( A^{(i)} \right)_{i\in I},
\]
where $A^{(i)}=\left[a_{st}^{(i)}\right]_{n\times n}\in \mathbf{M}_n(S_i)$, $i\in I$.

Now let $B=[b_{st}]_{n\times n}$ and $C=[c_{st}]_{n\times n}$ be arbitrary matrices in $\mathbf{M}_n\bigl(\prod_{i\in I}S_i\bigr)$.
Then
\[
\begin{aligned}
\varphi(B+C)
&=\left((B+C)^{(i)}\right)_{i\in I}
 =\left(B^{(i)} + C^{(i)}\right)_{i\in I} \\
&=\left(B^{(i)}\right)_{i\in I}+\left(C^{(i)}\right)_{i\in I}
 =\varphi(B)+\varphi(C).
\end{aligned}
\]
Moreover, we have
\[
(BC)_{st}
=\sum_{k=1}^{n}b_{sk}c_{kt}=\sum_{k=1}^{n}\left(b_{sk}^{(i)}\right)_{i\in I}\left(c_{kt}^{(i)}\right)_{i\in I}
=\sum_{k=1}^{n}\left(b_{sk}^{(i)}c_{kt}^{(i)}\right)_{i\in I}
=\left(\sum_{k=1}^{n}b_{sk}^{(i)}c_{kt}^{(i)}\right)_{i\in I},
\]
which implies that
\[
(BC)^{(i)}=\left[\sum_{k=1}^{n}b_{sk}^{(i)}c_{kt}^{(i)}\right]_{n\times n}
= \left[b_{st}^{(i)}\right]_{n\times n}\left[c_{st}^{(i)}\right]_{n\times n} =B^{(i)}C^{(i)}.
\]
Thus
\[
\varphi(BC)=\left((BC)^{(i)}\right)_{i\in I}=\left(B^{(i)}C^{(i)}\right)_{i\in I}
=\left(B^{(i)}\right)_{i\in I}\left(C^{(i)}\right)_{i\in I}=\varphi(B)\varphi(C).
\]
Consequently, $\varphi$ is a semiring homomorphism.

Finally, we show that $\varphi$ is bijective.
Indeed, if $\varphi(B)=\varphi(C)$, then $B^{(i)}=C^{(i)}$ for all $i$,
and so $b_{st}^{(i)}=c_{st}^{(i)}$ for all $s, t$, and $i$.
This shows that $b_{st}=c_{st}$ for all $s$ and $t$,
and so $B=C$. We have proved that $\varphi$ is injective.
Now let $\left(Y^{(i)}\right)_{i\in I}$ be an arbitrary element of $\prod_{i\in I}\mathbf{M}_n(S_i)$
with $Y^{(i)}=\left[y_{st}^{(i)}\right]_{n\times n}$.
Take a matrix $D=[d_{st}]_{n\times n}$ such that $d_{st}=\left(y_{st}^{(i)}\right)_{i\in I}$.
Then $D$ is a matrix in $\mathbf{M}_n(\prod_{i\in I}S_i)$, and
\[
\varphi(D)=\left(D^{(i)}\right)_{i\in I}=\left(\left[y_{st}^{(i)}\right]_{n\times n}\right)_{i\in I}
=\left(Y^{(i)}\right)_{i\in I}.
\]
Hence $\varphi$ is surjective, and so $\varphi$ is bijective.
Therefore, $\varphi$ is a semiring isomorphism.
We conclude that
$\mathbf{M}_n\bigl(\prod_{i\in I}S_i\bigr)$ is isomorphic to $\prod_{i\in I}\mathbf{M}_n(S_i)$.
\end{proof}

\begin{lem}\label{lem25}
Let $S_1$ and $S_2$ be ai-semirings, and let $n\geq 2$ be an integer.
If $S_1$ is a subsemiring of $S_2$,
then $\mathbf{M}_n(S_1)$ is a subsemiring of $\mathbf{M}_n(S_2)$.
\end{lem}
\begin{proof}
The verification is routine; we omit the details.
\end{proof}

\begin{lem}\label{lem26}
Let $S_1$ and $S_2$ be ai-semirings, and let $n\geq 2$ be an integer.
If $S_1$ is a homomorphic image of $S_2$,
then $\mathbf{M}_n(S_1)$ is a homomorphic image of $\mathbf{M}_n(S_2)$.
\end{lem}
\begin{proof}
Suppose that $\varphi\colon S_1 \to S_2$ is a semiring epimorphism. Consider the mapping
\[
\psi \colon \mathbf{M}_n(S_1) \to \mathbf{M}_n(S_2), \quad [a_{ij}]_{n\times n} \mapsto [\varphi(a_{ij})]_{n\times n}.
\]
It is a routine matter to check that $\psi$ is a semiring epimorphism.
\end{proof}

We now state the first theorem on matrix semirings.

\begin{thm}\label{thm27}
Let $S_1$ and $S_2$ be ai-semirings, and let $n\geq 2$ be an integer.
If $\mathsf{V}(S_1)=\mathsf{V}(S_2)$, then $\mathsf{V}(\mathbf{M}_n(S_1))=\mathsf{V}(\mathbf{M}_n(S_2))$.
\end{thm}
\begin{proof}
This follows immediately from Lemmas~\ref{lem24}, \ref{lem25}, and \ref{lem26}.
\end{proof}

The second theorem is an embedding theorem,
which allows us to compare the varieties $\mathsf{V}(\mathbf{M}_n(S))$ for different $n$.

\begin{thm}[Embedding Theorem]\label{thm1}
Let $S$ be an ai-semiring, and let $n\geq 2$ be an integer.
Then the matrix semiring $\mathbf{M}_n(S)$ can be embedded into $\mathbf{M}_{n+1}(S)$.
\end{thm}
\begin{proof}
Define a mapping $\varphi\colon \mathbf{M}_{n}(S) \to \mathbf{M}_{n+1}(S)$ by
\[
\varphi(A)=
\begin{bmatrix}
a_{11} & a_{12} & \cdots & a_{1n} & a_{1n}\\
a_{21} & a_{22} & \cdots & a_{2n} & a_{2n}\\
\vdots & \vdots & \ddots & \vdots & \vdots\\
a_{n1} & a_{n2} & \cdots & a_{nn} & a_{nn}\\
a_{n1} & a_{n2} & \cdots & a_{nn} & a_{nn}
\end{bmatrix},
\]
where $A=[a_{ij}]_{n\times n}$ with $a_{ij}\in S$;
that is, $\varphi(A)$ is obtained from $A$ by appending a copy of the last column and a copy of the last row,
and setting the new diagonal entry to $a_{nn}$.
It is straightforward to see that $\varphi$ is injective and preserves addition.

It remains to show that $\varphi$ preserves multiplication.
The verification mainly relies on the idempotence of addition in $S$.
Let $A=[a_{ij}]_{n\times n}$ and $B=[b_{ij}]_{n\times n}$ be arbitrary matrices in $\mathbf{M}_{n}(S)$.
For any $1 \leq i, j \leq n$, we consider four cases based on the positions of $i$ and $j$.

\textbf{Case 1.} $1 \leq i, j \leq n$. Then
\[
(\varphi(A)\varphi(B))_{ij} = \left(\sum_{k=1}^{n} a_{ik} b_{kj}\right) + a_{in} b_{nj}
=\sum_{k=1}^{n} a_{ik} b_{kj}=(AB)_{ij}=(\varphi(AB))_{ij}.
\]

\textbf{Case 2.} $1 \leq i \leq n$, $j = n+1$. Then
\[
(\varphi(A)\varphi(B))_{i, n+1} = \left(\sum_{k=1}^{n} a_{ik} b_{kn}\right) + a_{in} b_{nn}
=\sum_{k=1}^{n} a_{ik} b_{kn}= (AB)_{in}=(\varphi(AB))_{i, n+1}.
\]

\textbf{Case 3.} $i = n+1$, $1 \leq j \leq n$. Then
\[
(\varphi(A)\varphi(B))_{n+1, j} = \left(\sum_{k=1}^{n} a_{nk} b_{kj}\right) + a_{nn} b_{nj}
=\sum_{k=1}^{n} a_{nk} b_{kj} = (AB)_{nj}=(\varphi(AB))_{n+1, j}.
\]

\textbf{Case 4.} $i = j = n+1$. Then
\begin{align*}
(\varphi(A)\varphi(B))_{n+1,n+1}
&=\left(\sum_{k=1}^{n} a_{nk} b_{kn}\right) + a_{nn} b_{nn} =\sum_{k=1}^{n} a_{nk} b_{kn}\\
&=(AB)_{nn}=(\varphi(AB))_{n+1, n+1}.
\end{align*}

This shows that $\varphi$ preserves multiplication.
Hence $\varphi$ is a semiring monomorphism,
and so $\mathbf{M}_{n}(S)$ can be embedded into $\mathbf{M}_{n+1}(S)$.
\end{proof}

Let $\mathcal{V}$ and $\mathcal{W}$ be varieties of ai-semirings.
Then $\mathcal{W}$ is a \textit{subvariety} of $\mathcal{V}$, denoted by $\mathcal{W} \leq \mathcal{V}$,
if $\mathcal{W}$ is contained in $\mathcal{V}$.
Moreover, $\mathcal{W}$ is a \textit{proper subvariety} of $\mathcal{V}$, denoted by $\mathcal{W} < \mathcal{V}$, if $\mathcal{W}$ is a subvariety of $\mathcal{V}$ and $\mathcal{W} \neq \mathcal{V}$.
\begin{cor}\label{coro26052420}
Let $S$ be an ai-semiring. Then
\begin{equation}\label{2026052401}
\mathsf{V}(S) \leq \mathsf{V}(\mathbf{M}_{2}(S)) \leq \cdots \leq\mathsf{V}(\mathbf{M}_{n}(S))\leq \mathsf{V}(\mathbf{M}_{n+1}(S))\leq\cdots.
\end{equation}
\end{cor}
\begin{proof}
This follows directly from Theorem~\ref{thm1}.
\end{proof}

\cite[Propositions 3.2 and 3.3 ]{jr2025} illustrate that the ascending chain \eqref{2026052401}
may already stabilize at $\mathsf{V}(S)$ or $\mathsf{V}(\mathbf{M}_2(S))$.
In contrast, the following result shows that the chain \eqref{2026052401} can be strictly increasing from the very beginning.

\begin{pro}\label{prop:chain}
The following varieties form an infinite strictly ascending chain:
\[
\mathsf{V}(D_2) < \mathsf{V}(\mathbf{M}_2(D_2)) < \cdots < \mathsf{V}(\mathbf{M}_n(D_2)) < \mathsf{V}(\mathbf{M}_{n+1}(D_2)) < \cdots .
\]
\end{pro}
\begin{proof}
By Corollary~\ref{coro26052420},
it suffices to show that $\mathsf{V}(D_2)$ is properly contained in $\mathsf{V}(\mathbf{M}_2(D_2))$,
and that for each $n \geq 2$ there exists an identity that holds in $\mathsf{V}(\mathbf{M}_n(D_2))$ but fails in $\mathsf{V}(\mathbf{M}_{n+1}(D_2))$.

It is easy to see that the identity $x \approx x^2$ holds in $D_2$ but fails in $\mathbf{M}_2(D_2)$; hence $\mathsf{V}(D_2) < \mathsf{V}(\mathbf{M}_2(D_2))$.

From \cite[Formula (37)]{sg1996} (where it is called the Euler-Fermat formula) or from the results in \cite[Lemma~2.1, Theorem~3.2, Corollary~4.1]{hl2004}, one can deduce that $\mathbf{M}_n(D_2)$ satisfies the identity
\begin{equation}\label{eq26052501}
x^{(n-1)^2+1} \approx x^{(n-1)^2+1+[n]},
\end{equation}
where $[n]$ denotes the least common multiple of $1, 2, \dots, n$.

It remains to show that this identity fails in $\mathbf{M}_{n+1}(D_2)$.
To this end, we need to find a matrix $A$ in $\mathbf{M}_{n+1}(D_2)$ such that $A^{(n-1)^2+1} \neq A^{(n-1)^2+1+[n]}$.
Let $A_{n+1}$ be the matrix in $\mathbf{M}_{n+1}(D_2)$ defined by
\[
A_{n+1} = \begin{bmatrix}
0 & 1 & 0 & \cdots & 0 \\
0 & 0 & 1 & \cdots & 0 \\
\vdots & \vdots & \vdots & \ddots & \vdots \\
1 & 0 & 0 & \cdots & 1 \\
1 & 0 & 0 & \cdots & 0
\end{bmatrix},
\]
where the entries are $1$ exactly on the superdiagonal ($(i, i+1)$ for $1 \leq i \leq n$), at $(n, 1)$, and at $(n+1, 1)$;
all other entries are $0$.
If $n=2$, then it is easy to check that $A_{3}^{2}\neq A_{3}^{4}$.
Now suppose that $n\geq 3$.
By \cite[Example 3.4]{Dd1999},
one can obtain that $A_{n+1}^k=[1]_{n+1}$ for all $k\geq n^2+1$, but $A_{n+1}^k\neq [1]_{n+1}$ for all $1\leq k\leq n^2$,
where $[1]_{n+1}$ denotes the constant $(n+1)\times (n+1)$ matrix with all entries equal to $1$.
Since $(n-1)^2+1< n^2$, it follows that $A_{n+1}^{(n-1)^2+1}\neq [1]_{n+1}$.
Since $2n-1\leq n(n-1)\leq [n]$, we obtain that $(n-1)^2+1+[n]\geq n^2+1$,
and so $A_{n+1}^{(n-1)^2+1+[n]}=[1]_{n+1}$.
Consequently, $A_{n+1}^{(n-1)^2+1} \neq A_{n+1}^{(n-1)^2+1+[n]}$.
Therefore, $\mathbf{M}_{n+1}(D_2)$ does not satisfy the identity \eqref{eq26052501}.
\end{proof}

\begin{remark}
Let $D$ be a nontrivial distributive lattice, and let $n\geq 2$ be an integer.
Then
\[
\mathsf{V}(\mathbf{M}_n(D))=\mathsf{V}(\mathbf{M}_n(D_2)).
\]
This follows from Theorem~\ref{thm27},
and $\mathsf{V}(D)=\mathsf{V}(D_2)$,
which relies on the fact that $\mathsf{V}(D_2)$ is the variety of all distributive lattices
and is a minimal nontrivial ai-semiring variety (see \cite{sr}).
\end{remark}

\section{Nonfinite basis property of the matrix semiring $\mathbf{M}_n(S_7)$}
In this section, we apply \cite[Theorem 2.2]{gjr25},
a criterion due to Gao et al. that is based on the work of Jackson et al.~\cite{jrz},
to show that the matrix semiring $\mathbf{M}_n(S_7)$ is nonfinitely based for all $n \geq 2$.
For completeness and the reader's convenience, we first provide some basic facts concerning hypergraphs.

A \emph{$3$-uniform hypergraph} $\mathbb{H}$ is a pair $(V, E)$,
where $E$ is a family of $3$-element subsets of a set $V$.
Each element of $V$ is a \emph{vertex} of $\mathbb{H}$, and each element of $E$ is a \emph{hyperedge} of $\mathbb{H}$.
Let $\mathbb{H}=(V, E)$ be a $3$-uniform hypergraph.
Then $\mathbb{H}$ is \emph{$2$-colourable} if there exists a mapping $\varphi: V \to \{0, 1\}$ such that for every hyperedge $e \in E$,
the image $\varphi(e)$ contains exactly two distinct values; that is, $|\varphi(e)| = 2$.

A \emph{cycle} of $\mathbb{H}$ is an alternating sequence $v_1, e_1, v_2, e_2, \ldots, v_n, e_n$ of
distinct vertices and hyperedges such that $v_1 \in e_1 \cap e_n$ and $v_{i+1} \in e_i \cap e_{i+1}$ for $1 \leq i < n$.
The length of this cycle is $n$. The \emph{girth} of $\mathbb{H}$ is the length of its shortest cycle.
Let $n \geq 2$ be an integer.
It follows from \cite[Theorem 2.7]{ham18}
that there exists a $3$-uniform hypergraph $\mathbb{H}_n= (V_n, E_n)$ that is not $2$-colourable and has girth greater than $3\binom{3n}{2}$.
Moreover, $\mathbb{H}_n$ has no isolated vertex; that is, every vertex of $\mathbb{H}_n$ lies in some hyperedge.

Now let $\{x_v \mid v \in V_n\}$ be a set of variables in one-to-one correspondence with $V_n$.
We shall use $\bt_{\mathbb{H}_n}$ to denote the ai-semiring term
\[
\sum_{\{u, v, w\} \in E_n} x_{u} x_{v} x_{w}.
\]
Note that each hyperedge of $\mathbb{H}_n$ gives rise to 3!=6 distinct words in $\bt_{\mathbb{H}_n}$.
This fact is crucial and will be used repeatedly in the proof of Theorem~\ref{thm26052701}.

The following result is due to Gao et al.~\cite[Theorem 2.2]{gjr25}.

\begin{lem}\label{lem26052501}
Let $\mathcal{V}$ be a variety of ai-semirings that contains the flat semiring $S_c(abc)$.
If for every $n \geq 2$ there exists a vertex $v_n \in V_n$ such that $\mathcal{V}$ satisfies the identity
\begin{equation}\label{id26052550}
\mathbf{t}_{\mathbb{H}_n} \approx \mathbf{t}_{\mathbb{H}_n} + x_{v_n},
\end{equation}
then $\mathcal{V}$ is nonfinitely based.
\end{lem}

We shall apply Lemma~\ref{lem26052501} to prove that $\mathbf{M}_n(S_7)$ is nonfinitely based.
As a first step, we establish some basic properties of $\mathbf{M}_n(S_7)$.
Some of these will be used repeatedly in the sequel without explicit reference.


\begin{pro}
Let $n \geq 2$ be an integer, and let $[\infty]_n$, $[a]_n$,
and $[1]_n$ denote the constant $n \times n$ matrices in $\mathbf{M}_n(S_7)$
with all entries equal to $\infty$, $a$, and $1$, respectively.
\begin{enumerate}[label=(\arabic*), font=\normalfont]
\item $[\infty]_n$ is the additive maximum element and the multiplicative zero element of $\mathbf{M}_n(S_7)$;
\item $[a]_n$ and $[1]_n$ are both additive minimal elements of $\mathbf{M}_n(S_7)$;
\item If $AB = [1]_n$ for some $A, B \in \mathbf{M}_n(S_7)$, then $A = B = [1]_n$.
\end{enumerate}
\end{pro}

\begin{pro}\label{pro3.33}
Let $n \geq 2$ be an integer, let $A$ and $B$ be matrices in $\mathbf{M}_n(S_7)$, and let $1 \leq i, j \leq n$.
\begin{enumerate}[label=(\arabic*), font=\normalfont]
\item $(AB)_{ij} = 1$ if and only if $A_{ik} = B_{kj} = 1$ for all $1 \leq k \leq n$;
that is, the $i$-th row of $A$ consists entirely of $1$'s, and the $j$-th column of $B$ consists entirely of $1$'s.

\item $(AB)_{ij} = a$ if and only if $\{A_{ik}, B_{kj}\} = \{a, 1\}$ for all $1 \leq k \leq n$.
\end{enumerate}
\end{pro}

\begin{proof}
The result follows directly from the Cayley tables of $S_7$ (see Table~\ref{tbs7}) together with the definition of matrix multiplication.
In $S_7$, a sum equals $1$ only when every summand is $1$, and a product equals $1$ if and only if both factors are $1$.
Similarly, a sum equals $a$ only when every summand is $a$, and a product equals $a$ if and only if one factor is $1$ and the other is $a$.
\end{proof}

\begin{cor}\label{lem3.7}
Let $n\geq 2$ be an integer, and let $A, B, C$ be matrices in $\mathbf{M}_n(S_7)$ such that $AB=C$.
If every row or every column of $C$ contains at least one entry equal to $1$, then $A =[1]_n$ or $B =[1]_n$.
In particular, if $C$ has a row consisting entirely of $1$'s, then $B =[1]_n$;
if $C$ has a column consisting entirely of $1$'s, then $A =[1]_n$.
\end{cor}
\begin{proof}
This is a direct consequence of Proposition~\ref{pro3.33}.
\end{proof}

\begin{pro}\label{lem1}
Let $n \geq 2$ be an integer, and let $A$, $B$ and $C$ be matrices in $\mathbf{M}_n(S_7)$.
If $(ABC)_{ij} = 1$ for some $1 \leq i, j \leq n$, then $B = [1]_n$.
\end{pro}

\begin{proof}
Suppose $(ABC)_{ij} = 1$ for some $1 \leq i, j \leq n$. Then
\[
(ABC)_{ij} = \sum_{k=1}^{n} \sum_{l=1}^{n} A_{ik} B_{kl} C_{lj} = 1.
\]
Since the sum equals $1$, every summand must be $1$; in particular, $B_{kl} = 1$ for all $1 \leq k, l \leq n$. Hence $B = [1]_n$.
\end{proof}

\begin{pro}\label{lem2}
Let $n \geq 2$ be an integer, and let $A \in \mathbf{M}_n(S_7)$.
If $([1]_n A [1]_n)_{ij} = a$ for some $1 \leq i, j \leq n$, then $A = [a]_n$.
\end{pro}

\begin{proof}
Since $([1]_n A [1]_n)_{ij} = a$, we have
\[
a = \sum_{k=1}^{n} \sum_{l=1}^{n} ([1]_n)_{ik} A_{kl} ([1]_n)_{lj}
   = \sum_{k=1}^{n} \sum_{l=1}^{n} 1 \cdot A_{kl} \cdot 1
   = \sum_{k=1}^{n} \sum_{l=1}^{n} A_{kl}.
\]
Since $a$ is an additive minimal element of $S_7$, the sum equals $a$ only if every summand is $a$. Thus $A_{kl} = a$ for all $1 \leq k, l \leq n$; that is, $A = [a]_n$.
\end{proof}

Let $\mathcal{V}_1$ and $\mathcal{V}_2$ be varieties of ai-semirings such that $\mathcal{V}_1\leq \mathcal{V}_2$.
The interval $[\mathcal{V}_1, \mathcal{V}_2]$ denotes the class of all subvarieties of $\mathcal{V}_2$ that contain $\mathcal{V}_1$.
From \cite[Figure 1]{gjr25} we know that $\mathsf{V}(S_7)$ contains $S_c(abc)$,
and since $\mathsf{V}(\mathbf{M}_n(S_7))$ contains $\mathsf{V}(S_7)$,
the same holds for $\mathsf{V}(\mathbf{M}_n(S_7))$.
Hence the notation $[\mathsf{V}(S_c(abc)), \mathsf{V}(\mathbf{M}_n(S_7))]$ is well-defined.

\begin{thm}\label{thm26052701}
Let $n\geq 2$ be an integer.
Every variety in $[\mathsf{V}(S_c(abc)), \mathsf{V}(\mathbf{M}_n(S_7))]$ is nonfinitely based.
\end{thm}
\begin{proof}
For each $n \geq 2$, let us fix a vertex $v_n$ of the hypergraph $\mathbb{H}_n$.
We first show that $\mathbf{M}_n(S_7)$ satisfies the identity \eqref{id26052550}.
Indeed, let $\varphi: \{x_v \mid v\in V_n\} \to \mathbf{M}_n(S_7)$ be an arbitrary substitution.
Then
\[
\varphi(\bt_{\mathbb{H}_n}) = \sum_{\{u, v, w\} \in E_n} \varphi(x_u x_v x_w)
= \sum_{\{u, v, w\} \in E_n} \varphi(x_u)\varphi(x_v)\varphi(x_w),
\]
which is denoted by $[T_{ij}]_{n\times n}$
with $T_{ij} \in S_7$ for $1 \leq i, j \leq n$.
For each vertex $s \in V_n$, we write $A^{(s)}$ for $\varphi(x_s)$.
We consider the following three cases.

\textbf{Case 1.} $T_{ij} = \infty$ for all $1 \leq i, j \leq n$.
Then $\varphi(\bt_{\mathbb{H}_n})=[\infty]_n$, and so
\[
\varphi(\bt_{\mathbb{H}_n} + x_{v_n})=\varphi(\bt_{\mathbb{H}_n}) + \varphi(x_{v_n}) = [\infty]_n+ \varphi(x_{v_n})= [\infty]_n,
\]
since $[\infty]_n$ is the additive maximum element of $\mathbf{M}_n(S_7)$.
Consequently, $\varphi(\bt_{\mathbb{H}_n})=\varphi(\bt_{\mathbb{H}_n} + x_{v_n})$.

\textbf{Case 2.} $T_{ij} = 1$ for some $1 \leq i, j \leq n$.
Then the $(i,j)$-entry of $\varphi(x_u)\varphi(x_v)\varphi(x_w)$ equals $1$ for every $\{u, v, w\} \in E_n$,
since $1$ is an additive minimal element of $S_7$.
Hence $(A^{(u)}A^{(v)}A^{(w)})_{ij} = 1$ for all $\{u, v, w\} \in E_n$.
By Proposition~\ref{lem1}, it follows that $A^{(v)}=[1]_n$.
Since $\{v, w, u\}$ and $\{w, u, v\}$ are also hyperedges in $E_n$, the same argument yields $A^{(w)} = [1]_n$ and $A^{(u)} = [1]_n$.
This shows that $\varphi(x_u)=\varphi(x_v)=\varphi(x_w) =[1]_n$ for all $\{u, v, w\} \in E_n$,
and so
\[
\varphi(x_ux_vx_w)=\varphi(x_u)\varphi(x_v)\varphi(x_w)=[1]_n[1]_n[1]_n=[1]_n
\]
for all $\{u, v, w\} \in E_n$.
Consequently, $\varphi(\bt_{\mathbb{H}_n})=[1]_n$.
Therefore,
\[
\varphi(\bt_{\mathbb{H}_n} + x_{v_n})=\varphi(\bt_{\mathbb{H}_n}) + \varphi(x_{v_n}) = [1]_n+ [1]_n= [1]_n=\varphi(\bt_{\mathbb{H}_n}).
\]

\textbf{Case 3.} $T_{ij} = a$ for some $1 \leq i, j \leq n$.
Then
\[
(A^{(u)}A^{(v)}A^{(w)})_{ij}=(\varphi(x_u)\varphi(x_v)\varphi(x_w))_{ij}=(\varphi(x_ux_vx_w))_{ij}=a
\]
for all $\{u, v, w\} \in E_n$,
since $a$ is an additive minimal element of $S_7$.
Consequently, every entry of each $A^{(v)}$ ($v \in V_n$) is either $1$ or $a$. Moreover,
\[
A^{(u)}_{ik} A^{(v)}_{kl} A^{(w)}_{lj} = a \quad \text{for all } 1 \leq k, l \leq n.
\]

We first show that for each $v \in V_n$, the diagonal entries of $A^{(v)}$ are all equal.
Indeed, fix $v \in V_n$ and choose a hyperedge $\{u, v, w\} \in E_n$ (such a hyperedge exists because $\mathbb{H}_n$ has no isolated vertex).
For any $1 \leq k \leq n$, we have
\[
A^{(u)}_{ik} A^{(v)}_{kk} A^{(w)}_{kj} = A^{(u)}_{ik} A^{(w)}_{kj} A^{(v)}_{jj} = a.
\]
Since the multiplicative reduct of $S_7$ is commutative and $0$-cancellative,
it follows that $A^{(v)}_{kk} = A^{(v)}_{jj}$, as required.

\textbf{Subcase 3.1.}
For any hyperedge $\{u, v, w\} \in E_n$, one of its vertices, say $w$, satisfies $A^{(w)}_{sj} = a$ for some $1 \leq s \leq n$.
Then $(A^{(u)}A^{(v)})_{is}=1$.
This implies that the $i$-th row of $A^{(u)}$ and the $s$-th column of $A^{(v)}$ consist entirely of $1$'s.
Furthermore, the diagonal entries of $A^{(v)}$ are all equal to $1$.

We now show that every entry of $A^{(v)}$ must be $1$.
Indeed, suppose that for some $k, l$ we have $A^{(v)}_{kl} = a$
(since all entries are either $1$ or $a$ in this case).
Then because the $i$-th row of $A^{(u)}$ consists entirely of $1$'s,
\[
(A^{(u)}A^{(v)})_{il} = \sum_{p=1}^{n} A^{(u)}_{ip} A^{(v)}_{pl}=\sum_{p=1}^{n} 1\cdot A^{(v)}_{pl}=\sum_{p=1}^{n}  A^{(v)}_{pl}=\infty,
\]
since $A^{(v)}_{kl} = a$ and $A^{(v)}_{ll} = 1$ (so the sum contains both $a$ and $1$, yielding $\infty$ in $S_7$).
This leads to $(A^{(u)}A^{(v)}A^{(w)})_{ij} = \infty$, a contradiction.
Hence $A^{(v)}_{kl} \neq a$ for all $k, l$, and thus $A^{(v)}_{kl} = 1$ for all $k, l$. Therefore, $A^{(v)} = [1]_n$.

Similarly, from $A^{(w)}_{sj} = a$ we obtain $(A^{(v)}A^{(u)})_{is} = 1$, and by the same argument as above, $A^{(u)} = [1]_n$.
Since $(A^{(u)}A^{(w)}A^{(v)})_{ij}=a$, it now follows that $([1]_nA^{(w)}[1]_n)_{ij}=a$.
By Proposition~\ref{lem2}, $A^{(w)}=[a]_n$.
Thus, for every hyperedge $\{u, v, w\} \in E_n$, we have $\{A^{(u)}, A^{(v)}, A^{(w)}\} = \{[1]_n, [1]_n, [a]_n\}$.
Consequently, $\mathbb{H}_n$ is $2$-colourable, a contradiction.

\textbf{Subcase 3.2.}
There exists a hyperedge $\{u, v, w\} \in E_n$, $A^{(u)}_{sj}=A^{(v)}_{sj}=A^{(w)}_{sj} = 1$ for all $1 \leq s \leq n$.
Then the diagonal entries of $A^{(u)}$ and $A^{(v)}$ are all equal to $1$,
and $(A^{(u)}A^{(v)})_{is}=a$ for all $1 \leq s \leq n$.
In particular, $(A^{(u)}A^{(v)})_{ii}=a$.
This implies that $A^{(u)}_{ii}A^{(v)}_{ii}=a$, and so $A^{(u)}_{ii}=a$ or $A^{(v)}_{ii}=a$,
a contradiction.

Thus both subcases lead to contradictions; therefore Case 3 cannot occur.

Since Cases 1, 2, 3 exhaust all possibilities, we have shown that for every substitution $\varphi$,
\[
\varphi(\mathbf{t}_{\mathbb{H}_n}) = \varphi(\mathbf{t}_{\mathbb{H}_n} + x_{v_n}).
\]
Hence $\mathbf{M}_n(S_7)$ satisfies the identity \eqref{id26052550}.
By Lemma~\ref{lem26052501}, every variety in the interval
$[\mathsf{V}(S_c(abc)), \mathsf{V}(\mathbf{M}_n(S_7))]$ is nonfinitely based.
\end{proof}

\begin{cor}\label{cor26060320}
Let $n\geq 2$ be an integer.
Every variety in $[\mathsf{V}(S_7), \mathsf{V}(\mathbf{M}_n(S_7))]$ is nonfinitely based.
In particular, the following varieties are nonfinitely based and form an infinite ascending chain:
\begin{equation*}
\mathsf{V}(S_7) < \mathsf{V}(\mathbf{M}_2(S_7)) \leq \mathsf{V}(\mathbf{M}_3(S_7)) \leq \cdots \leq \mathsf{V}(\mathbf{M}_n(S_7)) \leq \mathsf{V}(\mathbf{M}_{n+1}(S_7)) \leq \cdots.
\end{equation*}
\end{cor}
\begin{proof}
This follows directly from Theorem~\ref{thm26052701}, Corollary~\ref{coro26052420},
and the fact that the identity $xy \approx yx$ is satisfied by $S_7$, but fails in $\mathbf{M}_2(S_7)$.
\end{proof}

To end this section, we apply Theorem~\ref{thm26052701} to solve the finite basis problem for several five-element ai-semirings.
The finite basis problem for ai-semirings of order at most four has been nearly settled,
with only two algebras remaining (see \cite{rlyc, rlzc, ryy, yrzs}).
Up to isomorphism, there are exactly $15751$ ai-semirings of order five (see \cite{edwards2025}),
which are denoted by $S_{(5, i)}$, $1 \leq i \leq 15751$.
We assume that the carrier set of each of these semirings is $\{1, 2, 3, 4, 5\}$.

Applying Theorem~\ref{thm26052701}, we obtain that $73$ five-element ai-semirings are nonfinitely based.
The verification for each of these algebras follows the same pattern as for the representative example $S_{(5,5158)}$, and is therefore omitted for brevity.

\begin{table}[ht]
\centering
\caption{The Cayley tables of $S_{(5,5158)}$} \label{tb545}
\begin{tabular}{c|ccccc}
$+$ &  $1$ & $2$ & $3$ & $4$ & $5$ \\
\hline
$1$ & $1$ & $1$ & $1$ & $1$ & $1$ \\
$2$ & $1$ & $2$ & $1$ & $1$ & $2$ \\
$3$ & $1$ & $1$ & $3$ & $3$ & $1$ \\
$4$ & $1$ & $1$ & $3$ & $4$ & $1$ \\
$5$ & $1$ & $2$ & $1$ & $1$ & $5$ \\
\end{tabular}\qquad
\begin{tabular}{c|ccccc}
$\cdot$ & $1$ & $2$ & $3$ & $4$ & $5$ \\
\hline
$1$ & $1$ & $1$ & $1$ & $1$ & $1$ \\
$2$ & $1$ & $1$ & $1$ & $2$ & $1$ \\
$3$ & $1$ & $1$ & $1$ & $1$ & $1$ \\
$4$ & $1$ & $1$ & $3$ & $4$ & $5$ \\
$5$ & $1$ & $1$ & $1$ & $5$ & $1$ \\
\end{tabular}
\end{table}
\begin{cor}
The five-element ai-semiring $S_{(5,5158)}$ is nonfinitely based.
\end{cor}

\begin{proof}
By Corollary~\ref{cor26060320}, it suffices to show that $\mathsf{V}(S_{(5,5158)})$ lies in the interval $[\mathsf{V}(S_7), \mathsf{V}(\mathbf{M}_2(S_7))]$.
Indeed, it is easy to check that $S_7$ is isomorphic to the subalgebra $\{1,4,5\}$ of $S_{(5,5158)}$, so $\mathsf{V}(S_7)$ is a subvariety of $\mathsf{V}(S_{(5,5158)})$.

Now consider the mapping $\varphi$ from $S_{(5,5158)}$ to the direct product of three copies of $\mathbf{M}_2(S_7)$ defined by
\[
\varphi(1)= \bigl([\infty]_2,\; [\infty]_2,\; [\infty]_2\bigr), \quad \varphi(2) = \bigl( A,\; [\infty]_2,\; [\infty]_2 \bigr),
\]
\[
\varphi(3) = \bigl( B,\; B,\; [\infty]_2 \bigr), \quad \varphi(4) = \bigl([1]_2,\; [1]_2,\; [1]_2\bigr),
\quad \varphi(5) = \bigl([a]_2,\; [a]_2,\; [a]_2\bigr),
\]
where
\[
A = \begin{bmatrix} a & a \\ \infty & \infty \end{bmatrix},\quad
B = \begin{bmatrix} 1 & \infty \\ 1 & \infty \end{bmatrix}.
\]
It is straightforward to verify that $\varphi$ is an embedding. Hence $\mathsf{V}(S_{(5,5158)})$ is a subvariety of $\mathsf{V}(\mathbf{M}_2(S_7))$.
\end{proof}

\section{The interval $[\mathsf{V}(S_7), \mathsf{V}(\mathbf{M}_2(S_7))]$}
Corollary~\ref{cor26060320} tells us that $\mathsf{V}(S_7)$ is a proper subvariety of $\mathsf{V}(\mathbf{M}_2(S_7))$.
This naturally raises the question: what is the size of the interval $[\mathsf{V}(S_7), \mathsf{V}(\mathbf{M}_2(S_7))]$?
In this section, we show that this interval contains at least countably infinitely many distinct varieties.

Let $\mathbf{w}$ be a word and $S$ an ai-semiring.
We say that $\mathbf{w}$ is \emph{minimal} for $S$ if
whenever $\mathbf{w} \approx \mathbf{w} + \mathbf{u}$ holds in $S$ for some word $\mathbf{u}$,
it follows that $\mathbf{u} = \mathbf{w}$.
This is equivalent to $\mathbf{w}$ being an \emph{isoterm} for $S$;
that is, $S$ satisfies no nontrivial identity of the form $\mathbf{w} \approx \mathbf{w}'$.
The following result is due to Ren et al. \cite[Lemma 5.6]{ren2023}.

\begin{lem}\label{iso}
Let $\mathbf{w}$ be a word and $S$ an ai-semiring.
Then $\bw$ is minimal for $S$ if and only if the flat semiring $S(\bw)$ belongs to the variety $\mathsf{V}(S)$.
\end{lem}

A \emph{linear} word is a word without repeated variables.
\begin{pro}\label{pro26053150}
Every linear word is minimal for $\mathbf{M}_2(S_7)$.
\end{pro}
\begin{proof}
Let $\bv = x_1 x_2 \cdots x_m$ be an arbitrary linear word.
Suppose that $\bu$ is a word such that the identity $\bv \approx \bv+\bu$ holds in $\mathbf{M}_2(S_7)$.
For any substitution $\varphi : \{x_1, x_2, \ldots, x_m\} \to \mathbf{M}_2(S_7)$, we have that $\varphi(\bu) \leq \varphi(\bv)$.
For a word $\bw$, let $c(\bw)$ denote the set of all variables occurring in $\bw$.
Since $S_7$ can be embedded into $\mathbf{M}_2(S_7)$,
it follows from \cite[Proposition 5.5]{jrz} that $c(\bu) \subseteq c(\bv)$.
Conversely, if $c(\mathbf{v}) \nsubseteq c(\mathbf{u})$, then there exists $x_j \in c(\mathbf{v})$ such that $x_j \notin c(\mathbf{u})$.
Consider the substitution $\varphi_j$ defined by $\varphi_j(x_j) = [a]_2$ and $\varphi_j(x) = [1]_2$ for all other variables $x$.
Then $\varphi_j(\mathbf{u}) = [1]_2$ and $\varphi_j(\mathbf{v}) = [a]_2$, which would imply $[1]_2 \leq [a]_2$, a contradiction.
Thus $c(\bv) \subseteq c(\bu)$ and so $c(\bu) = c(\bv)$.

Next, we show that $\mathbf{u}$ is a linear word.
Indeed, suppose that some variable $x_j$ occurs at least twice in $\mathbf{u}$.
Then $\varphi_j(\mathbf{u}) = [\infty]_2$ and $\varphi_j(\mathbf{v}) = [a]_2$, which would imply $[\infty]_2 \leq [a]_2$, a contradiction.
Hence $\mathbf{u}$ is linear.

By the above observations, we may assume that $\mathbf{u} = x_{i_1} x_{i_2} \cdots x_{i_m}$,
where $\{i_1, \dots, i_m\}$ is a permutation of $\{1, \dots, m\}$.
We proceed by induction on $j$ to show that $i_j = j$ for all $1 \leq j \leq m$.

First, suppose that $i_1 \neq 1$.
Define a substitution $\varphi\colon \{x_1, x_2, \ldots, x_m\} \to \mathbf{M}_2(S_7)$ by
\[
\varphi(x_1) = \begin{bmatrix} 1 & 1 \\ a & a \end{bmatrix}, \quad
\varphi(x) = [1]_2 \text{ for } x \neq x_1.
\]
Then $\varphi(\mathbf{v}) = \begin{bmatrix} 1 & 1 \\ a & a \end{bmatrix}$ and $\varphi(\mathbf{u}) = [\infty]_2$, which would imply $[\infty]_2 \leq \begin{bmatrix} 1 & 1 \\ a & a \end{bmatrix}$, a contradiction. Hence $i_1 = 1$.

Now assume that $i_1 = 1, \dots, i_{j-1} = j-1$ for some $j$ with $2 \leq j \leq m$,
and suppose that $i_j \neq j$.
Define a substitution $\varphi\colon \{x_1, x_2, \ldots, x_m\} \to \mathbf{M}_2(S_7)$ by
\[
\varphi(x_{j-1}) = \begin{bmatrix} 1 & a \\ 1 & a \end{bmatrix}, \quad
\varphi(x_j) = \begin{bmatrix} a & a \\ 1 & 1 \end{bmatrix}, \quad
\varphi(x) = [1]_2 \text{ for all other } x.
\]
Then $\varphi(\mathbf{u}) = [\infty]_2$ and $\varphi(\mathbf{v}) = [a]_2$, yielding $[\infty]_2 \leq [a]_2$, a contradiction. Thus $i_j = j$.

By induction, $i_j = j$ for all $1 \leq j \leq m$. Therefore $\mathbf{u} = \mathbf{v}$, and consequently $\mathbf{v}$ is minimal for $\mathbf{M}_2(S_7)$.
\end{proof}

\begin{pro}\label{pro26060420}
The interval $[\mathsf{V}(S_7), \mathsf{V}(\mathbf{M}_2(S_7))]$ contains at least countably infinitely many distinct varieties.
\end{pro}
\begin{proof}
For each integer $k \geq 1$, let $\mathbf{w}_k$ denote the linear word $x_1 x_2 \cdots x_k$.
By Proposition~\ref{pro26053150} and Lemma~\ref{iso}, the flat semiring $S(\mathbf{w}_k)$ belongs to the variety $\mathsf{V}(\mathbf{M}_2(S_7))$.
Let $\mathcal{V}_k$ denote the join of $\mathsf{V}(S_7)$ and $\mathsf{V}(S(\mathbf{w}_k))$. Then $\mathcal{V}_k$ lies in the interval $[\mathsf{V}(S_7), \mathsf{V}(\mathbf{M}_2(S_7))]$.

Since $S(\mathbf{w}_k)$ embeds into $S(\mathbf{w}_{k+1})$, it follows that $\mathcal{V}_k \leq \mathcal{V}_{k+1}$.
Moreover, one can easily verify that the identity
\[
y_1 y_2 y_3 \cdots y_k y_{k+1} \approx y_2 y_1 y_3 \cdots y_k y_{k+1}
\]
holds in $\mathcal{V}_k$ but fails in $\mathcal{V}_{k+1}$,
because the multiplicative reduct of $S_7$ is commutative,
in $S(\mathbf{w}_k)$ the product of any $k+1$ elements is the multiplicative zero,
while in $S(\mathbf{w}_{k+1})$ the product $x_1x_2x_3\cdots x_kx_{k+1}\neq 0$, but $x_2x_1x_3\cdots x_kx_{k+1}=0$.
Hence $\mathcal{V}_k$ is a proper subvariety of $\mathcal{V}_{k+1}$.

Thus we obtain an infinite strictly ascending chain
\[
\mathcal{V}_1 < \mathcal{V}_2 < \cdots < \mathcal{V}_k < \mathcal{V}_{k+1} < \cdots
\]
in the interval $[\mathsf{V}(S_7), \mathsf{V}(\mathbf{M}_2(S_7))]$, which proves the required result.
\end{proof}

\begin{cor}
Let $n\geq 2$ be an integer.
The interval $[\mathsf{V}(S_7), \mathsf{V}(\mathbf{M}_n(S_7))]$ contains at least countably infinitely many distinct varieties.
\end{cor}
\begin{proof}
This follows from Proposition~\ref{pro26060420} and Corollary~\ref{cor26060320}.
\end{proof}

\section{The multiplicative reduct of $\mathbf{M}'_n(S_7)$}
Although the finite basis problem for $\mathbf{M}_n(S_7)$ ($n \geq 2$) has been solved,
it remains unknown whether $\mathsf{V}(\mathbf{M}_n(S_7)) = \mathsf{V}(\mathbf{M}_{n+1}(S_7))$ for all $n \geq 2$.
Let $\mathbf{M}'_n(S_7)$ denote the set $\mathbf{M}_n(S_7) \setminus \{[1]_n\}$.
It is easy to see that $\mathbf{M}'_n(S_7)$ forms a subsemiring of $\mathbf{M}_n(S_7)$.
In this section, we shall show that the multiplicative reduct of $\mathbf{M}'_n(S_7)$ is a $5$-nilpotent semigroup;
that is, the product of arbitrary five matrices in $\mathbf{M}'_n(S_7)$ is $[\infty]_n$.
This suggests that $\mathsf{V}(\mathbf{M}_n(S_7)) = \mathsf{V}(\mathbf{M}_{n+1}(S_7))$ for all $n \geq 2$;
that is, the ascending chain in Corollary~\ref{cor26060320} may stabilize already at $\mathsf{V}(\mathbf{M}_2(S_7))$.

\begin{pro}\label{pro26041302}
Let $n \geq 2$ be an integer, and let $A$ be a matrix in $\mathbf{M}_n(S_7)$.
\begin{enumerate}[label=(\arabic*), font=\normalfont]
\item
If $A_{ij}=\infty$ for some $1\leq i, j \leq n$,
then for every matrix $B \in \mathbf{M}_n(S_7)$,
every entry of the $i$-th row of $AB$ is $\infty$.

\item If $A_{ij}=\infty$ for some $1\leq i, j \leq n$,
then for every matrix $B \in \mathbf{M}_n(S_7)$,
every entry of the $j$-th column of $BA$ is $\infty$.
\end{enumerate}
\end{pro}
\begin{proof}
Suppose that $A_{ij}=\infty$ for some $1\leq i, j \leq n$. Then
\[
(AB)_{is}=\sum_{k=1}^{n} A_{ik}B_{ks}=\infty
\]
for all $1\leq s\leq n$;
that is,
every entry of the $i$-th row of $AB$ is $\infty$. Similarly, We also have every entry of the $j$-th column of $BA$ is $\infty$.
\end{proof}

As a consequence, we have
\begin{cor}\label{coro26060350}
Let $n \geq 2$ be an integer, and let $A$ be a matrix in $\mathbf{M}_n(S_7)$.
\begin{itemize}
\item[$(1)$] If every row of $A$ contains an entry equal to $\infty$,
then for every matrix $B \in \mathbf{M}_n(S_7)$, $AB=[\infty]_n$.

\item[$(2)$] If every column of $A$ contains an entry equal to $\infty$,
then for every matrix $B \in \mathbf{M}_n(S_7)$, $BA=[\infty]_n$.
\end{itemize}
\end{cor}

The following result describes the annihilators of the multiplicative reducts of $\mathbf{M}_n(S_7)$.
\begin{cor}\label{cor3.8}
Let $n \geq 2$ be an integer, and let $A$ be a matrix in $\mathbf{M}_n(S_7)$.
Then $AB=BA=[\infty]_n$ for all $B\in \mathbf{M}_n(S_7)$
if and only if every row and column of $A$ contains an entry equal to $\infty$.
\end{cor}
\begin{proof}
If every row and column of $A$ contains an entry equal to $\infty$,
then by Corollary~\ref{coro26060350}, $AB=BA=[\infty]_n$ for all $B\in \mathbf{M}_n(S_7)$.

Conversely, suppose $AB = BA = [\infty]_n$ for all $B \in \mathbf{M}_n(S_7)$.
Assume, for contradiction, that there exists $1\leq i \leq n$ such that the $i$-th row of $A$ contains no entry equal to $\infty$.
Then each entry in this row is either $1$ or $a$.
Take a matrix $B$ in $\mathbf{M}_n(S_7)$ with
\[
B_{k1} =
\begin{cases}
1 & \text{if } A_{ik} = a, \\
a & \text{if } A_{ik} = 1,
\end{cases}
\quad \text{for } k = 1, \dots, n,
\]
and all other entries equal to $\infty$.
Then
\[
(AB)_{i1} = \sum_{k=1}^{n} A_{ik} B_{k1} = \sum_{k=1}^{n} a=a,
\]
and so $AB \neq   [\infty]_n$, a contradiction.
Hence every row of $A$ must contain an entry equal to $\infty$.
A symmetric argument using $BA = [\infty]_n$ shows that every column of $A$ also contains an entry equal to $\infty$.
\end{proof}

\begin{cor}\label{lem3}
Let $n \geq 2$ be an integer, and let $A$ and $B$ be matrices in $\mathbf{M}_n(S_7)$.
If there exists $1 \leq i \leq n$ such that for all $1 \leq j \leq n$,
\[
(AB)_{ij} = (BA)_{ij} = a,
\]
then $A$ and $B$ are constant matrices. More precisely, either $A=[1]_n$ and $B=[a]_n$, or $A=[a]_n$ and $B=[1]_n$.
\end{cor}
\begin{proof}
Assume that for some fixed index $i$ and for every $j$,
$(AB)_{ij} = (BA)_{ij} = a$. Then
\[
(AB)_{ij} = \sum_{k=1}^{n} A_{ik}B_{kj} = a,\quad
(BA)_{ij} = \sum_{k=1}^{n} B_{ik}A_{kj} = a.
\]
In $S_7$, a sum equals $a$ only if every summand equals $a$. Hence, for all $j$ and each $k$,
\[
A_{ik}B_{kj} = a \quad \text{and} \quad B_{ik}A_{kj} = a;
\]
that is,
\begin{equation}\label{id26060405}
A_{ik}B_{k1} = A_{ik}B_{k2} = \dots = A_{ik}B_{kn} = a
\end{equation}
and
\begin{equation}\label{id26060406}
B_{ik}A_{k1} = B_{ik}A_{k2}= \cdots = B_{ik}A_{kj}=a.
\end{equation}
Since the multiplicative reduct of $S_7$ is $0$-cancellative, it follows from \eqref{id26060405} that
\[
B_{k1} = B_{k2} = \dots = B_{kn}
\]
for each $k$; that is, $B$ is a row-constant matrix.
Furthermore, from $A_{il}B_{ll} = a$ and $B_{ii}A_{il} = a$ for each $l$, we obtain
\[
B_{11} = B_{22} = \dots = B_{nn}.
\]
Thus $B$ is a constant matrix.
Similarly, from \eqref{id26060406} one can show that $A$ is also a constant matrix.
By Corollary~\ref{cor3.8}, we conclude that $A \neq [\infty]_n$ and $B \neq [\infty]_n$.
Moreover, it is easy to see that $A \neq B$.
Therefore, either $A = [1]_n$ and $B = [a]_n$, or $A=[a]_n$ and $B=[1]_n$.
\end{proof}


\begin{pro}\label{prop3.8}
Let $n\geq 2$ be an integer, and let $A$ be a matrix in $\mathbf{M}'_n(S_7)$.
Then every row and column of $A^2$ contains at least one entry equal to $\infty$.
\end{pro}
\begin{proof}
For any $1 \leq i \leq n$, we have that $(A^2)_{ii} = \sum_{k=1}^{n} A_{ik}A_{ki}.$
Since $A_{ii}A_{ii}=\infty$ or $1$, it follows immediately that $(A^2)_{ii} =\infty$ or $1$;
that is, the diagonal entries of $A^2$ are either $\infty$ or $1$.

If $(A^2)_{ii} = \infty$, then the $i$-th row and $i$-th column of $A^2$ both contain at least one entry equal to $\infty$.
Now suppose that $(A^2)_{ii}=1$.
Then by Proposition~\ref{pro3.33}, $A_{ik} = A_{ki} = 1$ for all $1\leq k\leq n$.
Assume, for contradiction, that $(A^2)_{ij} \neq \infty$ for all $1\leq j \leq n$. Then
\[
(A^2)_{ij} = \sum_{k=1}^{n} A_{ik} A_{kj} =\left(\sum_{k=1}^{n} A_{ik} A_{kj}\right)+ A_{ii} A_{ij}
=\left(\sum_{k=1}^{n} A_{ik} A_{kj}\right)+1\cdot 1
=1
\]
for each $j$,
which by Corollary~\ref{lem3.7} forces $A =[1]_n$. This contradicts the fact that $A \in \mathbf{M}'_n(S_7)$.
Hence the $i$-th row of $A^2$ contains at least one entry equal to $\infty$.
Similarly, the $i$-th column of $A^2$ contains at least one entry equal to $\infty$.
This completes the proof.
\end{proof}

\begin{cor}
Let $n \geq 2$ be an integer.
Then $[1]_n$ and $[\infty]_n$ are the only idempotent matrices in $\mathbf{M}_n(S_7)$.
\end{cor}
\begin{proof}
Let $A$ be an idempotent matrix in $\mathbf{M}_n(S_7)$ with $A \neq [1]_n$.
Then $A \in \mathbf{M}'_n(S_7)$ and $A = A^2$.
By Proposition~\ref{prop3.8}, every row of $A$ contains an entry equal to $\infty$.
It then follows from Corollary~\ref{coro26060350} that $A^2 = [\infty]_n$, and consequently $A = [\infty]_n$.
Hence $[1]_n$ and $[\infty]_n$ are the only idempotent matrices in $\mathbf{M}_n(S_7)$.
\end{proof}

Recall that a \emph{nil-semigroup} is a semigroup $S$ with zero $0$ such that for every element $a \in S$,
there exists a positive integer $n$ such that $a^n = 0$.

\begin{cor}\label{cor3.9}
Let $n \geq 2$ be an integer.
The multiplicative reduct of $\mathbf{M}'_n(S_7)$ is a nil-semigroup. More precisely, $A^3 = [\infty]_n$ for every $A \in \mathbf{M}'_n(S_7)$, but $A^2 \neq [\infty]_n$ for some $A \in \mathbf{M}'_n(S_7)$.
\end{cor}

\begin{proof}
Let $A$ be a matrix in $\mathbf{M}'_n(S_7)$. By Proposition~\ref{prop3.8},
every row of $A^2$ contains at least one entry equal to $\infty$.
Then by Corollary~\ref{coro26060350}, it follows that $A^3 = [\infty]_n$.

To see that $A^2 \neq [\infty]_n$ for some $A\in \mathbf{M}'_n(S_7)$, consider the matrix
\[
A = \begin{bmatrix}
1 & 1 & \cdots & 1 & 1 \\
1 & 1 & \cdots & 1 & 1 \\
\vdots & \vdots & \ddots & \vdots & \vdots \\
1 & 1 & \cdots & 1 & 1 \\
1 & 1 & \cdots & 1 & a
\end{bmatrix}.
\]
One readily verifies that $A^2 \neq [\infty]_n$.
\end{proof}

It is well-known that every finite nil-semigroup is nilpotent.
We now determine the precise nilpotency index for the multiplicative reduct of $\mathbf{M}'_n(S_7)$.
\begin{pro}
Let $n\geq 2$ be an integer.
The multiplicative reduct of $\mathbf{M'}_n(S_7)$ is $5$-nilpotent, but is not $4$-nilpotent.
\end{pro}
\begin{proof}
Let $A, B, C, D, E$ be arbitrary matrices in $\mathbf{M'}_n(S_7)$.
By Proposition~\ref{lem1}, every entry of the product $ABCDE$ is either $\infty$ or $a$.
Suppose that $(ABCDE)_{ij}=a$ for some $1\leq i, j\leq n$. Then
\[
\left((ABC)(DE)\right)_{ij} = \sum_{k=1}^{n} (ABC)_{ik} (DE)_{kj} = a.
\]
This implies that $\{(ABC)_{ik}, (DE)_{kj}\}=\{a, 1\}$ for all $1\leq k\leq n$.
Combining this with Proposition~\ref{lem1}, one can deduce that $(ABC)_{ik}=a$ for all $1\leq k\leq n$,
and so $(DE)_{kj}=1$ for all $1\leq k\leq n$.
By Corollary~\ref{lem3.7}, $D=[1]_n$, contradicting $D \in \mathbf{M}'_n(S_7)$.
Thus no such $i, j$ exist, and therefore $ABCDE = [\infty]_n$.
Therefore, the multiplicative reduct of $\mathbf{M'}_n(S_7)$ is $5$-nilpotent.

To see that it is not $4$-nilpotent, consider the matrices
\[
A = \begin{bmatrix}
1 & 1 & \cdots & 1 \\
\infty & \infty & \cdots & \infty \\
\vdots & \vdots & \ddots & \vdots \\
\infty & \infty & \cdots & \infty
\end{bmatrix},\quad
B = \begin{bmatrix}
1 & a & \cdots & a \\
1 & a & \cdots & a \\
\vdots & \vdots & \ddots & \vdots \\
1 & a & \cdots & a
\end{bmatrix},
\]
\[
C = \begin{bmatrix}
a & a & \cdots & a \\
1 & 1 & \cdots & 1 \\
\vdots & \vdots & \ddots & \vdots \\
1 & 1 & \cdots & 1
\end{bmatrix},\quad
D = \begin{bmatrix}
1 & a & \cdots & a \\
1 & a & \cdots & a \\
\vdots & \vdots & \ddots & \vdots \\
1 & a & \cdots & a
\end{bmatrix}.
\]
One can easily check that $(ABCD)_{11} = a$, so $ABCD \neq [\infty]_n$.
Hence the multiplicative reduct is not $4$-nilpotent.
\end{proof}

\section{Conclusion}
We have proved that the matrix semiring $\mathbf{M}_n(S_7)$ is nonfinitely based for every $n \geq 2$, and obtained the following infinite ascending chain:
\[
\mathsf{V}(S_7) < \mathsf{V}(\mathbf{M}_2(S_7)) \leq \mathsf{V}(\mathbf{M}_3(S_7)) \leq \cdots \leq \mathsf{V}(\mathbf{M}_n(S_7)) \leq \mathsf{V}(\mathbf{M}_{n+1}(S_7)) \leq \cdots .
\]
We have further shown that every variety in the interval $[\mathsf{V}(S_7), \mathsf{V}(\mathbf{M}_n(S_7))]$ is nonfinitely based,
and that this interval contains at least countably infinitely many distinct varieties.
Thus, while Problem~\ref{prob26052301} remains open in general,
our results confirm that it has an affirmative answer for all varieties lying in this interval.
A full resolution of the problem, however, still seems far from being achieved.
Nevertheless, the results obtained in this paper represent a modest step towards a complete solution.

Moreover, two natural questions remain open: whether this interval contains uncountably many distinct varieties, and whether $\mathsf{V}(\mathbf{M}_n(S_7)) = \mathsf{V}(\mathbf{M}_{n+1}(S_7))$ holds for all $n \geq 2$. The fact that the multiplicative reduct of $\mathbf{M}'_n(S_7)$ is a $5$-nilpotent semigroup strongly suggests that the latter equality may indeed hold for all $n \geq 2$.

For an arbitrary ai-semiring $S$, it is natural to study the finite basis problem for the matrix semiring $\mathbf{M}_n(S)$ ($n \geq 2$). However, general guiding results are still lacking. The findings presented in this paper, together with those in \cite{jr2025} and the related work of Dolinka~\cite{dolinka2009} and Gusev and Volkov~\cite{gv2510}, serve as a first step in this direction.

Finally, it is natural to ask for which $S$ the ascending chain
\[
\mathsf{V}(S) \leq \mathsf{V}(\mathbf{M}_{2}(S)) \leq \cdots \leq\mathsf{V}(\mathbf{M}_{n}(S))\leq \mathsf{V}(\mathbf{M}_{n+1}(S))\leq\cdots
\]
stabilizes, and at which point it becomes stable.

\subsection*{Acknowledgment}
The authors would like to thank Zidong Gao and Mengya Yue
for their helpful discussions and contributions to this work.
We are also grateful to Professor Mikhail Volkov for his valuable
suggestions and comments on this paper.
Miaomiao Ren, corresponding author, is supported by National Natural Science Foundation of China (Grant Nos. 12371024, 12571020).

\subsection*{Conflict of Interest}
On behalf of all authors, the corresponding author states that there is no conflict of interest.

\bibliographystyle{amsplain}


\end{document}